\documentclass[11pt]{article}
\usepackage[T1]{fontenc}
\usepackage{lmodern}
\usepackage{amsmath,amssymb,amsthm,mathtools}
\usepackage[margin=1in]{geometry}
\usepackage{microtype,booktabs,enumitem,placeins}
\usepackage{pgfplots}
\pgfplotsset{compat=1.18}
\usepackage[hidelinks]{hyperref}
\hypersetup{pdftitle={The Subtractive Divisor Orbit: Frequent Large Jumps in the Hitting Time},
 pdfauthor={Marco Mantovanelli},pdfsubject={Arithmetic dynamics and the divisor function}}
\newtheorem{theorem}{Theorem}[section]
\newtheorem{lemma}[theorem]{Lemma}
\newtheorem{proposition}[theorem]{Proposition}
\newtheorem{corollary}[theorem]{Corollary}
\newtheorem{conjecture}[theorem]{Conjecture}
\newtheorem{hypothesis}[theorem]{Hypothesis}
\theoremstyle{definition}
\newtheorem{definition}[theorem]{Definition}
\theoremstyle{remark}

\newcommand{\one}{\mathbf 1}
\newcommand{\li}{\operatorname{li}}
\numberwithin{equation}{section}
\title{The Subtractive Divisor Orbit:\\
Frequent Large Jumps in the Hitting Time}
\author{Marco Mantovanelli\\
Independent Researcher, Germany\\
\texttt{marco@mantovanelli.de}}
\date{}
\begin{document}
\maketitle

\begin{abstract}
Let $a(n)$ be the number of iterations of $F(m)=m-\tau(m)$ needed to
reach zero, where $\tau(m)$ counts the positive divisors of $m$.
Write $\Delta(n)=a(n)-a(n-1)$. We prove that for every $\varepsilon>0$
and all sufficiently large $X$, each of the inequalities
\[
 \Delta(n)\ge X^{1/2-\varepsilon},
 \qquad
 \Delta(n)\le -X^{1/2-\varepsilon}
\]
holds for at least $X^{1-\varepsilon}$ integers $1\le n\le X$.
The proof uses the identity $F(2m)=2F(m)$ at odd $m$ and the fact
that an odd orbit first changes parity at a square.

For every fixed $C>0$, a positive proportion of starting values
$n\in(X,4X]$ satisfy
\[
 \left|a(n)-\frac{Cn}{\log n}\right|\gg_C \frac{\sqrt X}{\log X}.
\]
For the hitting time itself we prove
\[
 a(x)\ll_b \frac{x}{(\log x)^b}
 \qquad (b<\log 2),
\]
together with a lower bound having logarithmic exponent $2.988206\ldots$.
On complete crossings of $(N,2N]$, with $Z=2N$,
the logarithm of the step size has normal order $(\log 2)\log\log Z$,
uniformly in the starting value. Computations through
$10^9$ give the largest positive and negative first differences in this range.
\end{abstract}

\section{Introduction}
Starting with a positive integer $x$, subtract the number of divisors
of the current value at each step:
\begin{equation}\label{eq:orbit}
 n_0(x)=x,\qquad F(n)=n-\tau(n),\qquad n_{j+1}(x)=F(n_j(x)).
\end{equation}
Since $1\le\tau(n)\le n$, the orbit decreases strictly until it reaches
zero. Its hitting time and the change caused by increasing the starting
value by one are
\begin{equation}\label{eq:hitting-time}
 a(0)=0,\qquad a(x)=1+a(F(x)),\qquad
 \Delta(x)=a(x)-a(x-1)\quad(x\ge1).
\end{equation}
The hitting times and their first differences are recorded in OEIS
A155043 and A264893 \cite{OEIS_A155043,OEIS_A264893}. The main result
concerns how often the difference is large, in each direction.
All logarithms are natural.

\begin{theorem}[Frequency of large first differences]\label{thm:frequency}
For every $\varepsilon>0$ there exists $X_0(\varepsilon)$ such that, for
every integer $X\ge X_0(\varepsilon)$ and each $\sigma\in\{-1,1\}$,
\[
 \#\{1\le n\le X:\ \sigma\Delta(n)\ge X^{1/2-\varepsilon}\}
       \ \ge\ X^{1-\varepsilon}.
\]
More precisely, with $L=\log(2X)$ there are absolute constants $c,C>0$
such that, for all sufficiently large $X$,
\begin{align*}
 \#\{1\le n\le X:|\Delta(n)|\ge c\sqrt X/L^{12}\}&\ge CX/L^{12},\\
 \#\{1\le n\le X:\sigma\Delta(n)\ge c\sqrt X/L^{20}\}&\ge CX/L^{20}
       \qquad(\sigma=\pm1).
\end{align*}
\end{theorem}

In particular $\sum_{n\le X}|\Delta(n)|^2\gg X^2/(\log(2X))^{36}$.
An elementary identity explains why estimates for the divisor function
enter a question about two adjacent starting values. For odd $m$ we
have $F(2m)=2F(m)$. Thus $D(m)=a(2m)-a(m)$ is constant along the
initial odd part of an orbit up to its first square. There are only
$O(\sqrt X)$ possible values of $D(m)$ for odd $m\le X$.
The proof uses this restriction twice: first to bound the maximum of
$|\Delta|$ from below, then to count distinct large differences.
Karttunen had already identified the equality between the number of
square visits and the number of parity changes \cite{OEIS_A262680}.

The same restriction gives an obstruction to accurate smooth approximations.
For $f(t)=Ct/\log t$ or $C\li(t)$, with any fixed $C>0$, a positive
proportion of $n\in(X,4X]$ satisfy
$|a(n)-f(n)|\gg_C\sqrt X/\log X$ (Theorem~\ref{thm:jump-model}).
This is compatible with a relative asymptotic on the scale $x/\log x$.
Dirichlet's mean-value formula
\[
 \sum_{n\le x}\tau(n)=x\log x+(2\gamma-1)x+O(\sqrt x)
\]
motivates the following question \cite{Apostol1976}.
\begin{conjecture}\label{conj:main}
As $x\to\infty$ through the integers, $a(x)\asymp x/\log x$.
\end{conjecture}
Here $\asymp$ denotes two-sided comparison by positive constants. A
leading constant is a separate question. Classical results for
multiplicative functions on prescribed intervals and progressions
\cite{Shiu1980,MatomakiRadziwill2016} concern sampling sets fixed
independently of the values. In \eqref{eq:orbit} each value of $\tau$
determines where the next sample is taken.

\begin{theorem}[Bounds for the hitting time]\label{thm:length}
Let $L=\log x$, $s=\log L$ and $\lambda=\log 2$. For $q>1$ put
$B(q)=(2^q-1)/(q-1)$, let $q_*$ be its minimizer, and set
$\beta=B(q_*)$. There is an absolute constant $C>0$ such that
for all sufficiently large integers $x$
\[
  \frac{x}{L^\beta}\frac{s^{\,1+1/(2(q_*-1))}}{\log s}
   \ \ll\ a(x)\ \ll\
  \frac{x}{L^\lambda}\exp\!\big(C\sqrt{s\log s}\big).
\]
Numerically $q_*\approx2.1080461$ and $\beta\approx2.9882063$.
\end{theorem}

The bounds use a normal-order statement on the orbit itself. A complete
dyadic crossing consists of all positive states in $(N,2N]$ on an orbit
starting at $x\ge2N$. The length of a crossing equals its
displacement divided by its mean step.
\begin{theorem}[Normal order on the orbit]\label{thm:step-law}
Let $\mathcal B$ be the sequence of all positive states of an orbit
from $x$, with $Z=x$, or a complete crossing of $(N,2N]$, with $Z=2N$.
Write $V=|\mathcal B|$, $s=\log\log Z$, and $\lambda=\log 2$.
For every fixed $\delta>0$, as $Z\to\infty$,
\[
 \frac{1}{V}\#\{v\in\mathcal B:|\log\tau(v)-\lambda s|>\delta s\}
      \ll_\delta \frac{\log s}{s}.
\]
The estimate is uniform in the starting value of a complete crossing.
\end{theorem}

The average and typical step sizes can differ. If the mean on a
crossing is of order $\log N$, a vanishing fraction of its states must
carry asymptotically all its displacement. This distinguishes the
unconditional normal order from the mean needed in
Conjecture~\ref{conj:main}.

Section~\ref{sec:fluctuations} proves the frequency and approximation
theorems. Sections~\ref{sec:exact} and~\ref{sec:step-law} give the
finite bounds and the analytic estimates for the steps.
Sections~\ref{sec:regularity} and~\ref{sec:ladders} examine what additional
control of divisor incidence would suffice for an asymptotic formula.
Section~\ref{sec:numerics} gives exact values and numerical comparisons.

\section{Fluctuations between consecutive starting values}
\label{sec:fluctuations}

Write $\Delta(j)=a(j)-a(j-1)$ for $j\ge1$. The defining recurrence gives
\begin{equation}
   \sum_{F(n)<j\le n}\Delta(j)=1 \qquad(n\ge1).
   \label{eq:jump-compensation}
\end{equation}
This identity will allow us to pass from a lower bound for $|\Delta|$ to
separate lower bounds for the two signs. The other ingredient comes from
comparing the orbits starting at $n$ and $2n$.

\subsection{Doubling at odd states}
\begin{lemma}\label{lem:odd-doubling}
Put $D(n)=a(2n)-a(n)$ for $n\ge0$. For each odd $n$ there is an odd square
$s\le n$ on its orbit with $D(n)=D(s)$. In particular,
\[
    \#\{D(n):1\le n\le Y,\ n\ \mathrm{odd}\}
      \le \#\{u\ge1:u\ \mathrm{odd},\ u^2\le Y\}\le\sqrt Y .
\]
\end{lemma}
\begin{proof}
For odd $n$, multiplicativity gives $\tau(2n)=2\tau(n)$, hence
$F(2n)=2F(n)$. The recurrence yields $D(n)=D(F(n))$.
An odd nonsquare has even divisor count, so its successor is again odd.
The strictly decreasing orbit must therefore encounter an odd square
before it reaches zero. Iteration up to the first such square proves
the assertion, also when the first square is the initial state.
\end{proof}

We use the following observation about any finite list of real numbers.
If the list takes at most $M$ distinct values, the difference between
its last and first terms is a sum of at most $M-1$ increments taken
from distinct consecutive pairs of the original list. Indeed, whenever
a value occurs twice, delete the part of the walk between the two
occurrences. The retained edges are original edges in chronological
order. The process ends at a walk with no repeated vertex.

\begin{proposition}\label{prop:maximum-jump}
For integers $X\ge2$ write
\[
 H(X)=\max_{1\le j\le X}|\Delta(j)|,\qquad
 S_p(X)=\sum_{1\le j\le X}|\Delta(j)|^p .
\]
Then
\[
 H(X)\gg \frac{a(X)}{\sqrt X},\qquad
 S_p(X)\gg_p \frac{a(X)^p}{X^{(p-1)/2}}\quad(p>1).
\]
\end{proposition}
\begin{proof}
For odd $j$,
\begin{equation}
 D(j+2)-D(j)=\sum_{i=1}^{4}\Delta(2j+i)
                    -\sum_{i=1}^{2}\Delta(j+i).
 \label{eq:odd-D-increment}
\end{equation}
For odd $n\le X/2$, the list $D(1),D(3),\ldots,D(n)$ begins at
$D(1)=0$ and takes at most $\sqrt n$ values. Each consecutive pair
contributes at most $6H(X)$ in absolute value. The preceding
observation gives $|D(n)|\le6\sqrt n\,H(X)$. For even $n\ge2$ use
$|D(n)-D(n-1)|\le3H(X)$.

For the second estimate, the sum of the $p$-th powers of the absolute
increments in this list is $\ll_p S_p(X)$. Indeed, use
$|\sum_{i=1}^6 t_i|^p\le 6^{p-1}\sum_{i=1}^6 |t_i|^p$ in
\eqref{eq:odd-D-increment}. Each index occurs at most once in each
of the two families of sums as $j$ runs over odd integers. Applying
H\"older's inequality to the retained edges gives
$|D(n)|\ll_p n^{(1-1/p)/2}S_p(X)^{1/p}$, with the even case again
following by comparison with $n-1$.

Now write $j=2r+e$, $e\in\{0,1\}$. For $j\ge2$,
$a(j)-a(r)=D(r)+e\Delta(j)$, and the final step from $1$ to $0$
contributes $\Delta(1)$. Telescope from $X$ along successive integer
halves. The geometric sums of $\sqrt j$ and of $j^{(1-1/p)/2}$ are
$O(\sqrt X)$ and $O_p(X^{(1-1/p)/2})$, respectively.
\end{proof}

\subsection{Frequency of large jumps}
We give the proof of Theorem~\ref{thm:frequency} with explicit powers of
the logarithm. Throughout this proof, the only analytic estimate needed is
the classical second-moment bound
$\sum_{v\le X}\tau(v)^2\ll X(\log(2X))^3$. The positive orbit states
are distinct and their divisor counts sum to $X$, so Cauchy--Schwarz
also gives $a(X)\gg X/(\log(2X))^3$.

\begin{proof}[Proof of Theorem~\ref{thm:frequency}]
Fix a sufficiently large integer $X$. Set
\[
 L=\log(2X),\qquad P=\lceil L^8\rceil,\qquad
 T=\frac{a(X)}{50L}\gg \frac{X}{L^4},\qquad
 K=\frac{T}{100\sqrt X}.
\]
Mark each integer threshold $y$ for which $F(v)\le y<v\le X$ for
some $v$ with $\tau(v)>P$. The number of marked thresholds is at most
\begin{equation}
  \sum_{\substack{1\le v\le X\\ \tau(v)>P}}\tau(v)
     \le P^{-1}\sum_{v\le X}\tau(v)^2\ll X/L^5=o(T).
  \label{eq:marked-thresholds}
\end{equation}
For an unmarked $y\ge0$ and a starting value $u\in[y,X]$, the first
orbit state at most $y$ belongs to
$I_y=[\max(0,y-P),y]\cap\mathbb Z$. If $u>y$, the step crossing the
threshold has length at most $P$. The crossing takes at most $u-y$
steps, as the orbit decreases by at least one per step.

\emph{A jump of size at least $T$.}
Suppose $|\Delta(j)|\ge T$ for some $j\le X$. We have $j\ge T$ since
$0\le a(t)\le t$. There are $\gg T$ unmarked integers $y$ in
$[j-T/4,j-1]$ by \eqref{eq:marked-thresholds}. From either $j$ or
$j-1$, the orbit reaches $I_y$ in at most $T/4$ steps. The two elapsed
times differ by at most $T/4$, so the values of $a$ at the two landing
points differ by at least $3T/4$. At least one adjacent pair in $I_y$
therefore has $|\Delta|\ge3T/(4P)$. Each adjacent pair can occur in
at most $P+1$ such intervals. This produces $\gg T/P$ distinct jumps.

\emph{All jumps smaller than $T$.}
If $H(X)<T$, some odd $Y\le X/2$ satisfies $|D(Y)|\ge T$. Otherwise,
for each $r\le X/2$ we would have $|D(r)|\le 4T$, by comparing even
$r\ge2$ with $r-1$ and using $D(0)=0$. Every integer halving
$j\mapsto\lfloor j/2\rfloor$ would then change $a$ by at most $5T$.
The resulting bound $a(X)\le 5T(1+\log X/\log 2)<50LT$ contradicts
the definition of $T$.

Call an odd $q\le Y$ good if $q,2q$ are unmarked and the oscillations
of $a$ on both $I_q$ and $I_{2q}$ are at most $K$. Include $1$ among
the good indices in any event. The landing values at thresholds $1,2$
lie in $[0,2]$, where the oscillation is $1\le K$ for large $X$.
For any good $q$ and odd $q<n\le Y$ we have
\begin{equation}
                  |D(n)-D(q)|\le 2K+2(n-q).
                  \label{eq:good-jump-comparison}
\end{equation}
To see this, follow $n$ and $2n$ to their first states at most $q$
and $2q$, respectively. The remaining times differ from $a(q)$ and
$a(2q)$, respectively, by at most $K$. The elapsed times both belong
to $[0,2(n-q)]$.

List the good odd indices in increasing order, appending $Y$ if it is
absent. Among the corresponding values of $D$ there are at most
$\sqrt X$ distinct values by Lemma~\ref{lem:odd-doubling}. By deleting
loops, one can pass from $D(1)=0$ to $D(Y)$ using at most $\sqrt X$
original edges of this list. For such an edge from $q$ to $n$, write
$n-q=2(1+b)$, where $b$ counts the omitted odd indices strictly
between them. These sets of omitted indices are disjoint over retained
edges, and each omitted index is not good. If $B$ is the total number
of odd indices up to $Y$ that are not good, summation of
\eqref{eq:good-jump-comparison} yields
\[
             T\le |D(Y)|\le (2K+4)\sqrt X+4B .
\]
Since $T/\sqrt X\to\infty$ and $K=T/(100\sqrt X)$, we obtain $B\gg T$.
Only $O(X/L^5)=o(T)$ of these indices can have $q$ or $2q$ marked.
Each remaining one has an adjacent pair with $|\Delta|>K/P$ in
$I_q$ or $I_{2q}$. A given pair can witness at most $2(P+1)$ indices.
This gives $\gg T/P$ distinct jumps of magnitude at least $K/P$.

Both cases give, for an absolute $c>0$,
\begin{equation}
  \#\{1\le j\le X:|\Delta(j)|\ge c\sqrt X/L^{12}\}
       \gg X/L^{12}. \label{eq:many-unsigned-jumps}
\end{equation}
Among these $j$, at most $O(XL^3/P^2)=O(X/L^{13})$ satisfy
$\tau(j)>P$. Retain the more frequent sign among the other indices.
For each retained $j$, \eqref{eq:jump-compensation} forces a jump
of the opposite sign in $(j-\tau(j),j]$ of magnitude at least
$(|\Delta(j)|-1)/P$. Indeed, the remaining jumps on that interval
sum to $1-\Delta(j)$. Each opposite-sign index occurs in at most
$P$ such intervals. Thus for an absolute $c'>0$ and each
$\sigma\in\{-1,1\}$,
\begin{equation}
  \#\{1\le j\le X:\sigma\Delta(j)\ge c'\sqrt X/L^{20}\}
       \gg X/L^{20}. \label{eq:many-signed-jumps}
\end{equation}
For fixed $\varepsilon>0$ the powers of $L$ and the implicit
constants are absorbed by $X^\varepsilon$ in the count and the
threshold. This proves the stated frequency assertion.
\end{proof}

\begin{corollary}\label{cor:jump-variation}
Let $L=\log(2X)$. For $X$ sufficiently large,
\[
 \frac1X\sum_{j=1}^X|\Delta(j)|\gg\frac{\sqrt X}{L^{24}},
 \qquad
 \sum_{j=1}^X|\Delta(j)|^2\gg\frac{X^2}{L^{36}} .
\]
\end{corollary}
\begin{proof}
In \eqref{eq:many-unsigned-jumps}, multiply the count by the threshold
and by its square, respectively, and divide the first result by $X$.
\end{proof}

\subsection{An obstruction to a small error term}
The same restriction on $D$ bounds the accuracy of smooth models for
individual hitting times.

\begin{theorem}\label{thm:jump-model}
Let $f$ be continuously differentiable for all sufficiently large real
arguments, and put $G(t)=f(2t)-f(t)$. Suppose that $G'(t)\ge c/\log t$
for some $c>0$ and all large $t$. There are $c_1,c_2>0$ such that,
for all sufficiently large $X$,
\[
  \#\left\{j\in(X,4X]\cap\mathbb Z:
       |a(j)-f(j)|\ge c_1\frac{\sqrt X}{\log X}\right\}
       \ge c_2X .
\]
This applies to $f(t)=Ct/\log t$ and $f(t)=C\operatorname{li}(t)$
for each fixed $C>0$.
\end{theorem}
\begin{proof}
On the odd integers $X<n\le2X$ there are $O(\sqrt X)$ possible
values of $D(n)$ by Lemma~\ref{lem:odd-doubling}. For any one of
these values $d$, the derivative bound shows that the real arguments
$t\in(X,2X]$ with $|G(t)-d|\le2h$ lie in an interval of length
$O_c(h\log X)$. Hence the number of odd $n$ with $|D(n)-G(n)|\le2h$
is $O_c(\sqrt X(1+h\log X))$. Take $h=c_1\sqrt X/\log X$ with
$c_1$ sufficiently small. This leaves $\gg X$ odd $n$ with
$|D(n)-G(n)|>2h$. In each pair $\{n,2n\}$, at least one error
exceeds $h$ in magnitude since
\[
 D(n)-G(n)=(a(2n)-f(2n))-(a(n)-f(n)).
\]
The pairs are disjoint. For each of the two functions in the statement,
$G'(t)\sim C/\log t$, which verifies the hypothesis.
\end{proof}

\section{Displacement and finite bounds}\label{sec:exact}

\begin{lemma}\label{lem:exact-energy}
For $0\le r<t\le a(x)$ the displacement of an orbit segment is
\begin{equation}\label{eq:segment-energy}
 n_r(x)-n_t(x)=\sum_{j=r}^{t-1}\tau(n_j(x)).
\end{equation}
In particular, the full orbit has total step mass
\begin{equation}\label{eq:exact-energy}
 \sum_{j<a(x)}\tau(n_j(x))=x .
\end{equation}
\end{lemma}
\begin{proof}
Sum $n_j-n_{j+1}=\tau(n_j)$ and use $n_{a(x)}=0$ for the second assertion.
\end{proof}

\begin{proposition}\label{thm:unconditional}
For every integer $x\ge2$,
\[
 \frac{x^2}{\sum_{n\le x}\tau(n)^2}\le a(x)
 \le\frac{3x}{8}+\frac{3}{8}\pi(\sqrt x)+\frac{11}{8}.
\]
\end{proposition}
\begin{proof}
The first estimate follows from \eqref{eq:exact-energy}, Cauchy--Schwarz
and the distinctness of the positive states. We use the classical estimate
\begin{equation}\label{eq:second-moment}
 \sum_{n\le x}\tau(n)^2\ll x(\log(2x))^3 .
\end{equation}
Indeed $\tau(n)^2\le d_4(n)$, the number of ordered factorizations into
four positive integers: at a prime power the inequality is
$(e+1)^2\le\binom{e+3}{3}$. Summing $d_4$ and bounding three of the
four sums by harmonic sums gives \eqref{eq:second-moment}. For a more
precise mean-square estimate see \cite{JiaSankaranarayanan2014}.

Let $P$ count prime states, $S$ squares of primes, $U$ visits to $1$,
and $R$ all remaining states. Then $U\le1$,
$a(x)=P+S+U+R$ and $x\ge2P+3S+U+4R$. A prime is followed by its
difference with $2$. Among three consecutive odd integers spaced by
$2$ one is divisible by $3$, so the only run of three primes is
$7,5,3$. The prime $2$ is followed by $0$. Hence the $S+R$ possible
separators give $P\le 2(S+R)+3$. It follows that
\[
 3x-8a(x)\ge-2P+S-5U+4R\ge-3S-11.
\]
Finally $S\le\pi(\sqrt x)$.
\end{proof}

\begin{proposition}[Square visits and parity]\label{prop:parity}
The parity changes at exactly the positive square states. There are
at most $\lfloor\sqrt x\rfloor$ such changes, including the final
transition to zero when its source is $1$.
\end{proposition}
\begin{proof}
Pair divisors $d,n/d$. The count is odd precisely at a square, and
$F(n)\equiv n-\tau(n)\pmod2$. The states are distinct.
\end{proof}

\subsection{Complete dyadic crossings}\label{sec:dyadic}
Let $t_x(y)=\min\{j:n_j(x)\le y\}$ for $0\le y\le x$. For $2N\le x$ put
\[
 J_x(N)=\{j:t_x(2N)\le j<t_x(N)\},\quad
 V_x(N)=|J_x(N)|,\quad E_x(N)=\sum_{j\in J_x(N)}\tau(n_j).
\]
When $V_x(N)>0$ set $T_x(N)=E_x(N)/V_x(N)$.

\begin{lemma}\label{lem:hitting-boundary}
Uniformly for $1\le y\le x$, $n_{t_x(y)}=y+O(\sqrt y+1)$. In particular
\begin{equation}\label{eq:dyadic-mass}
 E_x(N)=N+O(\sqrt N+1).
\end{equation}
\end{lemma}
\begin{proof}
If $x=y$ there is no overshoot. Otherwise the previous state $v>y$
satisfies $v-y\le\tau(v)\le2\sqrt v$, whence $\sqrt v\le1+\sqrt{1+y}$.
The overshoot $y-F(v)$ is at most $\tau(v)$. Apply
\eqref{eq:segment-energy} between the two crossing times.
\end{proof}

\section{The distribution of steps on an orbit}\label{sec:step-law}

Write $\lambda=\log 2$. We give a single argument for two choices of a
segment $J$: the full positive orbit starting at $x$, with $Z=x$, and
the indices of the states in $(N,2N]$ on any orbit starting at
$x\ge 2N$, with $Z=2N$. Throughout this section set
\[
 L=\log Z,\qquad s=\log L,\qquad V=|J|,\qquad
 E=\sum_{j\in J}\tau(n_j).
\]
We take $Z$ sufficiently large. Both segments consist of consecutive
indices, their states are distinct positive integers at most $Z$, and
\begin{equation}\label{eq:segment-size-v3}
 E\asymp Z,\qquad V\gg Z/L^3 .
\end{equation}
For the full orbit $E=Z$ by telescoping. For the dyadic segment
$E=N+O(\sqrt N)$: when a step first crosses a level $u$, its source
$v$ satisfies $v-u\le \tau(v)\le 2\sqrt v$, so the landing point is within
$O(\sqrt u+1)$ of $u$. This also holds if $x=u$ and the initial point
is used. Finally $E^2\le V\sum_{v\le Z}\tau(v)^2\ll V ZL^3$.
All constants in \eqref{eq:segment-size-v3} are independent of the start.

There is already a short proof of sublinear growth. For the set
$S_x$ of positive states visited from $x$, Cauchy--Schwarz and
the displacement identity give
\[
 a(x)^2
 \le \left(\sum_{v\in S_x}\tau(v)\right)
     \left(\sum_{v\in S_x}\frac1{\tau(v)}\right)
 \le x\sum_{v\le x}\frac1{\tau(v)}
 \ll \frac{x^2}{\sqrt{\log x}}.
\]
For this last step and for the tail estimate below, we use the
mean-value bound recorded in \cite[p.~337, (2)]{HallTenenbaum1991}.
If $f$ is nonnegative and multiplicative, and
$f(p^\nu)\le A B^{\nu-1}$ at all primes and all $\nu\ge1$ for fixed
$A>0$ and $1\le B<2$, then uniformly for $t\ge2$,
\begin{equation}\label{eq:multiplicative-mean}
 \sum_{u\le t} f(u)
 \ll_{A,B} t\exp\!\left(\sum_{p\le t}\frac{f(p)-1}{p}\right).
\end{equation}
For $f=1/\tau$ one can take
$A=B=1$ and $f(p)=1/2$. Mertens' theorem now gives the displayed
estimate and hence $a(x)\ll x/(\log x)^{1/4}$.

\subsection{A moment bound at prescribed shifts}

The proof first counts possible step sequences within a residue class.
Conditioning on that class fixes the contribution of small primes to
the next step once the preceding steps are known.
Let $d(v)^2$ be the largest square dividing $v$. Call a visited state
$v$ admissible if $d(v)\le L^5$ and $\tau(v)\le L^5$. The number of
inadmissible states in $J$ is $O(Z/L^5)$, by
\[
 \#\{v\le Z:d(v)>L^5\}\le Z\sum_{d>L^5}d^{-2}
 \quad\text{and}\quad
 \#\{j\in J:\tau(n_j)>L^5\}\le E/L^5.
\]
Factor $v=g(v)u$ where $g(v)$ is the product of all exact prime powers
$p^e\parallel v$ with $e\ge2$. Thus $u$ is squarefree and coprime to
$g(v)$. Since $e\le3\lfloor e/2\rfloor$ for $e\ge2$,
$g(v)\le d(v)^3\le L^{15}$ at admissible states.

Put $P=L^{1/2}$, $b=\lceil100s\rceil$ and
$W=\prod_{p\le P}p^b$. We have $\log W=O(s P\log P)=o(L)$.
For admissible $v$ every prime exponent is less than $b$ for large $Z$:
if $e\ge2$ then $2^e\le g(v)\le L^{15}$. Thus $v\bmod W$ determines
all exponents at primes up to $P$. Among primes $p>P$ with $p^e\parallel v$
and $e\ge2$, the sum of the exponents is at most $30$, since
$\log P=s/2$ and $\log g(v)\le15s$. Their contribution to $\tau(v)$
therefore takes only an absolute bounded number of possible values.
Large primes of exponent one contribute a factor $2^a$ with
$0\le a\le5s/\lambda$.

Consider a window of $k$ consecutive admissible states with initial
state $n$. The states are $n-h_1,\ldots,n-h_k$, with
$0=h_1<\cdots<h_k\le kL^5$, where each offset is the sum of the preceding
steps. Fix an initial class $a\bmod W$. Given a prefix
$(h_1,\ldots,h_i)$, the next source has residue $a-h_i\bmod W$.
Its small-prime contribution to $\tau$ is therefore fixed whenever
this prefix is realized by an admissible window. Its large-prime
contribution is one of the bounded number of products above times
$2^\ell$ for $0\le\ell\le5s/\lambda$. Thus there are at most $C_1s$
possible extensions $h_{i+1}=h_i+\tau(n-h_i)$. Induction gives a list
$\mathcal H_a$ of at most $(C_1s)^{k-1}$ offset vectors containing
every admissible length-$k$ window starting in this class.

\begin{lemma}\label{lem:joint-shifts}
Let $1\le m\le 2s$ be an integer and $0\le h_1,\ldots,h_m\le L^6$
distinct integers. Define
\[
 y=Z^{1/(100(m+1)s)},\qquad
 \mathcal P=\{p\text{ prime}:L^7<p\le y\},\qquad
 U(n)=\sum_{i=1}^m\sum_{p\in\mathcal P}\one_{p\mid n-h_i}.
\]
Choose $n$ uniformly among the integers in $[1,Z]$ in any specified
residue class modulo $W$, and set $\mu=\sum_{p\in\mathcal P}m/p$.
For some absolute $c_0>0$,
\begin{equation}\label{eq:joint-shifts}
 \mathbb P\bigl(|U-\mu|\ge\kappa ms\bigr)
   \ll \exp(-c_0\kappa^2ms),\qquad \mu=m\bigl(s+O(\log s)\bigr),
\end{equation}
uniformly in $1\le m\le 2s$, $0<\kappa\le1$, the shifts and the residue class.
\end{lemma}

\begin{proof}
The estimate for $\mu$ follows from Mertens' theorem on reciprocal primes,
since $\log\log y=s-\log(100(m+1)s)$ and $\log\log(L^7)=O(\log s)$.
For each $p\in\mathcal P$ the $m$ residues $h_i\bmod p$ are distinct and
$p\nmid W$. Let $B_p$ indicate that $n$ belongs to one of these residues,
so $U=\sum B_p$. Write $M_a=\#\{1\le n\le Z:n\equiv a\bmod W\}$
and $\mathbb E_a$ for expectation in this class. For distinct primes
$p_1,\ldots,p_d\in\mathcal P$, set $D=p_1\cdots p_d$. Writing
$n=a+Wj$, each $p_i$ allows exactly $m$ classes for $j\bmod p_i$.
By the Chinese remainder theorem this gives $m^d$ classes modulo $D$.
Each occurs $M_a/D+O(1)$ times in the interval of $M_a$ indices, so
\[
 \mathbb E_a\prod_{i=1}^d B_{p_i}
   = \prod_{i=1}^d\frac m{p_i}+O(m^d/M_a).
\]
The constant in this count is absolute even when $D>M_a$, and
$M_a=Z/W+O(1)\asymp Z/W$ uniformly in $a$.

Let $Y$ be the sum of independent Bernoulli variables with the
corresponding probabilities $m/p$. It suffices to take $\kappa^2ms$
larger than an absolute constant. Set $r=2\lfloor\kappa^2ms/200\rfloor$.
In the expansion of $U^j$, each ordered $j$-tuple of primes with $d$
distinct entries contributes the product of the corresponding $d$
indicators, since $B_p^l=B_p$ for $l\ge1$. The corresponding monomial
for $Y$ has mean $\prod_{i=1}^d(m/p_i)$. There are at most $y^j$
ordered tuples, and $m^d\le m^j$, hence
$|\mathbb E_a U^j-\mathbb E Y^j|\ll (W/Z)(my)^j$ for $1\le j\le r$.
Expanding both powers about the same deterministic $\mu=\mathbb EY$ yields
\[
 \big|\mathbb E_a (U-\mu)^r-\mathbb E(Y-\mu)^r\big|
 \ll \frac{W}{Z}\sum_{j=1}^r\binom rj\mu^{r-j}(my)^j
 \le \frac{W}{Z}(\mu+my)^r .
\]
This error is $O(Z^{-1/2})$ uniformly. Indeed, $y\ge2s$ for large $Z$,
$\mu\le2ms$, and $r\log y\le L/10000$. The remaining terms in its
logarithm are bounded by $\log W+r\log(2m)=o(L)$ in addition to $-L$,
since $r=O(s^2)$ and $m\le2s$.

For $0<t\le1$, independence and $e^{\pm t}-1\mp t\le t^2$ give
$\mathbb E e^{\pm t(Y-\mu)}\le e^{\mu t^2}$. For even $r$ it follows
that $\mathbb E(Y-\mu)^r\le 2r!\,t^{-r}e^{\mu t^2}$.
With $t=\sqrt{r/(2ms)}$, $\mu\le2ms$, and $r!\le r^r$, division by
$(\kappa ms)^r$ bounds this last expression by
$2(e\sqrt{2r/(ms)}/\kappa)^r\le2(1/2)^r$.
Markov's inequality proves \eqref{eq:joint-shifts}: when $\kappa^2ms$
is large, $\kappa ms\ge1$ and $r\asymp\kappa^2ms$. The normalized moment
error is at most $O(Z^{-1/2})$, which is $o(\exp(-c_0\kappa^2ms))$
because $\kappa^2ms\le 2s^2=o(L)$. Bounded $\kappa^2ms$ is covered
by enlarging the absolute implied constant.
\end{proof}

\subsection{Normal order on visited states}

\begin{proof}[Proof of Theorem~\ref{thm:step-law}]
Fix $0<\delta<\lambda/2$. The assertion for larger $\delta$ follows
by inclusion. Partition $J$ into disjoint windows of
$k=\lfloor s/\log s\rfloor$ consecutive states, with a remainder of
size less than $k$. For fixed $m$ the primes $p\le Z$ outside
$\mathcal P$ have reciprocal sum $O_m(\log s)$. For any fixed $\eta>0$,
the number of integers $v\le Z$ divisible by at least
$u=\lceil\eta s\rceil$ distinct such primes is at most
\begin{equation}\label{eq:omitted-primes}
 \frac{Z}{u!}\big(O_m(\log s)\big)^u\ll_{m,\eta} Z/L^6 .
\end{equation}
To see this, sum $Z/(p_1\cdots p_u)$ over sets of $u$ distinct primes
and apply Stirling's formula. For admissible $v$ we have $g(v)\le L^{15}$.
Thus for each fixed $\xi>0$ the divisor bound $\tau(g)\ll_\xi g^\xi$
gives $\log\tau(g(v))/s\le15\xi+O_\xi(1/s)$ uniformly in $v$.
Taking the limit superior as $Z\to\infty$ and then letting $\xi\downarrow0$ gives
\[
 0\le\log\tau(v)-\lambda\omega(v)\le\log\tau(g(v))=o(s).
\]
Here $\omega(v)$ is the number of distinct prime factors.
Choose $\eta$ so that $\lambda\eta<\delta/8$. Outside
\eqref{eq:omitted-primes}, an admissible $v$ with
$\log\tau(v)>(\lambda+\delta)s$ has more than
$(1+\delta/(2\lambda))s$ prime factors in $\mathcal P$, for $Z$ large.
At a state with $\log\tau(v)<(\lambda-\delta)s$, the number in
$\mathcal P$ is less than $(1-\delta/\lambda)s$, since
$2^{\omega(v)}\le\tau(v)$.

Fix a sign and put $\kappa=\delta/(4\lambda)$. A window with at least
$m$ deviations of this sign, avoiding the discarded states, determines
an initial class $a$, a vector $\mathbf h\in\mathcal H_a$, and a set
$I$ of $m$ deviating positions. Now fix this triple. For every initial
value realizing it, the preceding bounds give $|U_I-\mu|\ge\kappa ms$,
where $U_I$ uses the prescribed shifts $(h_i)_{i\in I}$.
We bound the number of such initial values by counting \emph{all}
$n\equiv a\bmod W$ in $[1,Z]$ satisfying this last inequality.
Thus Lemma~\ref{lem:joint-shifts} is applied in the original residue
class, without conditioning on the event that the orbit follows the
selected vector. This inclusion is valid for every adaptive choice
of successive steps.

Sum over the vectors and positions within each class. Since
$\sum_{a\bmod W}M_a=\lfloor Z\rfloor$, the resulting bound for the
total number of possible starts is
\[
 \ll \sum_{a\bmod W}M_a\,|\mathcal H_a|\binom{k}{m}
                 e^{-c_0\kappa^2ms}
 \le Z(C_1s)^k\binom{k}{m}e^{-c_0\kappa^2ms}.
\]
Choose $m=m(\delta)$ with $c_0\kappa^2m>8$ before taking $Z\to\infty$.
Then $k\ge m$ eventually, $k\log(C_1s)=(1+o(1))s$ and
$\log\binom{k}{m}=O_m(\log s)$, giving $O_\delta(Z/L^6)$ starts.
All offsets are at most $kL^5\le L^6$.

Each initial integer is used at most once on the orbit. Windows meeting
an inadmissible state or an exception in \eqref{eq:omitted-primes}
contain $O_\delta(kZ/L^5)$ indices in total. Outside these windows and
those with $m$ selected deviations, each window has fewer than $m$
deviations of that sign. By \eqref{eq:segment-size-v3}, the total
number of deviations of either sign is therefore
\begin{equation}\label{eq:quant-step-law}
 \#\{j\in J:|\log\tau(n_j)-\lambda s|>\delta s\}
   \ll_\delta V/k+kZ/L^5+k\ll_\delta V\frac{\log s}{s}.
\end{equation}
This proves the result for both segment choices. With $Z=2N$ the use of
$\log\log N$ instead of $s=\log\log(2N)$ follows by applying
\eqref{eq:quant-step-law} with $\delta/2$.
\end{proof}

\subsection{Bounds for the stopping time}

\begin{lemma}\label{lem:tilted-tail}
For each fixed real $q>1$, uniformly in $H\ge1$ and all sufficiently
large $Z$,
\begin{equation}\label{eq:tilted-tail}
 \sum_{\substack{v\le Z\\\tau(v)>H}}\tau(v)
   \ll_q \frac{Z L^{\,2^q-1}}{H^{q-1}\sqrt s}.
\end{equation}
\end{lemma}

\begin{proof}
First restrict the $q$th moment to $h\le\tau(v)<2h$. In the factorization
$v=gu$ above, for fixed $g$ the band allows at most one value of
$j=\omega(u)$, since $\tau(v)=\tau(g)2^j$. Write $z=2^q$.
Let $\pi_j(t)$ count squarefree integers $u\le t$ with exactly $j$
distinct prime factors. Hardy and Ramanujan
\cite[Lemma~A, (2.21)]{HardyRamanujan1917} prove that there are
absolute constants $C_2,K_2$ such that for \emph{all} $t\ge2$ and
integers $j\ge1$,
\[
 \pi_j(t)\le K_2\frac{t}{\log t}
        \frac{(\log\log t+C_2)^{j-1}}{(j-1)!}.
\]
For $g\le\sqrt Z$ take
$t=Z/g$, so $\log t\ge L/2$ and $\log\log t\le s$.
Dropping the coprimality condition on $u$ gives the upper bound
\[
 \sum_{\substack{u\le Z/g\ {\rm squarefree}\\ \omega(u)=j}}z^j
 \ll_q \frac Z{gL}\,z\,\frac{(z(s+C_2))^{j-1}}{(j-1)!}
 \ll_q \frac{Z L^{z-1}}{g\sqrt s}.
\]
The second inequality follows from
$\sup_{l\ge0}t^l/l!\ll e^t/\sqrt t$ for $t\ge1$ by Stirling's formula.
The bound also holds for $j=0$. Multiply by $\tau(g)^q$ and sum over
$g$. The sum of $\tau(g)^qg^{-3/4}$ over integers $g$ whose prime
exponents are all at least two, including $g=1$, converges by its
Euler product.

The remaining range is $\sqrt Z<g\le Z$, since $gu=v\le Z$ and $u\ge1$.
Extend the inner sum to all positive integers $u\le Z/g$.
The multiplicative function $f(u)=z^{\omega(u)}$ satisfies
$f(p^\nu)=z$ for every $\nu\ge1$, so \eqref{eq:multiplicative-mean}
applies with $A=z$ and $B=1$. Mertens' theorem gives
$\sum_{u\le t}z^{\omega(u)}\ll_q t(1+\log t)^{z-1}$ uniformly in
$t\ge2$, and the same bound holds directly for $1\le t<2$. Since
$g^{-1}\le Z^{-1/8}g^{-3/4}$ for $g>\sqrt Z$, the whole large-$g$
contribution is $O_q(ZL^{z-1}Z^{-1/8})$.
The resulting bound $O_q(ZL^{z-1}/\sqrt s)$ for the $q$th moment on
$[h,2h)$ is independent of $h\ge1$. On this band,
$\tau(v)\le h^{1-q}\tau(v)^q$. Summing over $h=2^jH$, $j\ge0$,
introduces the factor $H^{1-q}\sum_{j\ge0}2^{j(1-q)}
=H^{1-q}/(1-2^{1-q})$, proving \eqref{eq:tilted-tail} uniformly in $H\ge1$.
\end{proof}

\begin{proof}[Proof of Theorem~\ref{thm:length}]
Take $Z=x$ and $V=a(x)$. For fixed $q>1$ put
$B(q)=(2^q-1)/(q-1)$, $\theta_q=1/(2(q-1))$, and
$H=C_3 L^{B(q)}/s^{\theta_q}$. By Lemma~\ref{lem:tilted-tail},
for $C_3=C_3(q)$ large enough, visited states with $\tau>H$ carry at
most $x/2$ of the displacement. Fix $\delta>0$ with
$\lambda+\delta<B(q)$. At most $O_\delta(V\log s/s)$ states satisfy
$\tau>L^{\lambda+\delta}$ by \eqref{eq:quant-step-law}. Splitting the
remaining mass at $L^{\lambda+\delta}$ gives
\[
 x/2\le V L^{\lambda+\delta}+O_\delta(HV\log s/s)
      \ll_q HV\log s/s .
\]
Here $\delta$ may be fixed in terms of $q$. This yields the stated
lower bound on taking $q=q_*$, the minimizer of $B(q)$.
For completeness, $(q-1)^2 B'(q)=2^q((q-1)\log2-1)+1$
increases strictly from $-1$ to infinity on $(1,\infty)$, so $q_*$
is unique.

For the upper bound set $H_0=\exp(\lambda s-K\sqrt{s\log s})$ and
$k=\lfloor s/\log s\rfloor$, with $K$ a sufficiently large absolute
constant. A state $v$ with $\tau(v)<H_0$ has
$\omega(v)<s-(K/\lambda)\sqrt{s\log s}$. Consequently an admissible
window of $k$ such states satisfies
$U-\mu\le-\kappa ks$ in Lemma~\ref{lem:joint-shifts} with $m=k$ and
$\kappa=K\sqrt{\log s/s}/(2\lambda)$, since $\mu=k(s+O(\log s))$.
The number of possible initial states of all these windows is at most
\[
 \ll x(C_1s)^k\exp(-c_0\kappa^2ks)\ll x/L^4.
\]
There are at most $x/H_0$ visited states with $\tau\ge H_0$. Using
the same disjoint window partition, the inadmissible states, the large
steps, the all-small windows just counted, and the last incomplete
window give
\[
 a(x)\ll kx/L^5+kx/H_0+kx/L^4+k
       \ll x\exp\bigl(-\lambda s+O(\sqrt{s\log s})\bigr).
\]
\end{proof}

\subsection{Where the displacement is carried}

\begin{corollary}\label{cor:mass-away}
Let $T=E/V$ on a complete dyadic segment. Fix $\delta>0$ and suppose
that $T\ge L^{\lambda+\delta}$. A set of
$O_\delta(V\log s/s)$ states carries at least
$(1-L^{-\delta/2})E$ displacement. Moreover
\[
 \frac1{VT^2}\sum_{j\in J}(\tau(n_j)-T)^2
       \gg_\delta \frac{s}{\log s}.
\]
For every fixed $0\le\eta_0<1$, the states satisfying
$|\tau(n_j)-T|\le\eta_0T$ carry only
$O_{\delta,\eta_0}(E\log s/s)$ displacement.
\end{corollary}

\begin{proof}
Use the states with $\tau(n_j)>L^{\lambda+\delta/2}$. Their number
satisfies the stated bound by \eqref{eq:quant-step-law}. All other
states carry at most $VL^{\lambda+\delta/2}\le L^{-\delta/2}E$.
Cauchy--Schwarz on the selected set then gives
$\sum_{j\in J}\tau(n_j)^2\gg_\delta E^2s/(V\log s)$. Subtracting the
squared mean proves the variance bound. For large $Z$ a state within
relative distance $\eta_0$ of $T$ belongs to the selected set and
carries at most $(1+\eta_0)T$. Summing gives the last assertion.
\end{proof}

In particular, on any family with $Z\to\infty$, a fixed positive
fraction of the displacement can come from steps within a fixed
relative distance $\eta_0<1$ of $T$ only if $T=L^{\lambda+o(1)}$.
Indeed \eqref{eq:quant-step-law} already gives
$T\ge(1-o(1))L^{\lambda-\epsilon}$ for every fixed $\epsilon>0$, and
Corollary~\ref{cor:mass-away} rules out $T\ge L^{\lambda+\epsilon}$
on a subsequence. This applies, in particular, to the saturating
near-ladders considered below.

\section{Dyadic means and divisor incidence}\label{sec:regularity}
For the orbit from $x$ consider the scales $N_k=x/2^{k+1}$, $k\ge0$.
Write $E,V,T$ for the displacement, number of states and mean step on
the current crossing. In all limiting statements about crossings we
take $N\to\infty$. 

\begin{proposition}[Local-to-global criterion]\label{thm:local-global}
Suppose $T_x(N_k)\asymp \log N_k$ uniformly on all $N_k\ge\sqrt x$.
Then $a(x)\asymp x/\log x$. If, more precisely,
$T_x(N_k)\sim b\log N_k$ uniformly there for a fixed $b>0$, then
$a(x)\sim x/(b\log x)$.
\end{proposition}
\begin{proof}
By \eqref{eq:dyadic-mass}, $V=E/T\asymp N_k/\log N_k$. The top block
gives the lower bound, and summing the upper bounds gives
$O(x/\log x)$ since $\log N_k\ge(\log x)/2$. The rest of the path starts
below $2\sqrt x$ and takes $O(\sqrt x)$ steps. For the asymptotic, the
first $K$ blocks, with $K$ fixed, contribute
$(1-2^{-K}+o(1))x/(b\log x)$. The remaining contribution is
$O(2^{-K}x/\log x+\sqrt x)$. Let $x\to\infty$ and then $K\to\infty$.
\end{proof}

Let $A(d)=\#\{j\in J_x(N):d\mid n_j\}$ and set
\begin{equation}\label{eq:incidence}
 \mathcal D_x(N)=\frac{1}{V\log N}
       \sum_{d\le\sqrt{2N}}|A(d)-V/d|.
\end{equation}
\begin{proposition}\label{prop:incidence-mean}
If $\mathcal D_x(N)=o(1)$ along a family of complete dyadic crossings,
then $T\sim\log N$ along that family.
There is also a parity-dependent form. Write $p$ for the proportion
of even states in the crossing, let $w_d=1/d$ for odd $d$, and
$w_d=2p/d$ for even $d$. If
$\sum_{d\le\sqrt{2N}}|A(d)-Vw_d|=o(V\log N)$, then
$T=(1/2+p+o(1))\log N$.
\end{proposition}
\begin{proof}
For $N<v\le2N$ put $D_v(Y)=\#\{d\mid v:d\le Y\}$. Pairing $d$ and
$v/d$ gives $2D_v(\sqrt N)\le\tau(v)\le 2D_v(\sqrt{2N})$.
The strict inequality $v>N$ ensures that the two divisors in each
pair counted on the left are distinct. At either endpoint $Y$ the
error in $\sum_{d\le Y}A(d)$ is bounded by the full absolute discrepancy.
Using $\sum_{d\le Y}1/d=\log Y+O(1)$ gives $E=(1+o(1))V\log N$.
For the parity-dependent weights the odd and even harmonic sums each
equal $\frac12\log Y+O(1)$. Thus the two endpoint sums both give
$E=(1/2+p+o(1))V\log N$ after doubling. The error is uniform for
$0\le p\le1$.
\end{proof}

\subsection{Truncation and a single step-size band}\label{sec:tails}
For fixed $q>1$ and large $N$, Lemma~\ref{lem:tilted-tail} gives
\[
 \sum_{\substack{j\in J_x(N)\\\tau(n_j)>H}}\tau(n_j)
 \ll_q\frac{N(\log(2N))^{2^q-1}}{H^{q-1}\sqrt{\log\log(2N)}} .
\]
In particular, for $H=(\log N)^\beta$, this
displacement is $o(N)$ on taking $q=q_*$. If a subset of the crossing
carries at least $cN$ displacement for fixed $c>0$, its remaining states
occupy $O(\log\log N)$ dyadic bands. Some band $t\le\tau<2t$ therefore
contains $r$ states with $tr\gg_c N/\log\log N$. This argument alone
leaves the location of the band undetermined.

\section{Nearly equal steps on a crossing}\label{sec:ladders}
Fix a nonempty dyadic crossing $J$ with parameters $E,V,T$.

\begin{definition}\label{def:dynamic-ladder}
For $0<\eta<1$, an ordered subsequence $m_1>\cdots>m_r$ of the states
in $J$ is an $\eta$-dynamic near-ladder with parameter $T$ if
$|\tau(m_i)-T|\le\eta T$ and
\[
 m_i-m_{i+1}=\tau(m_i)+s_i,\quad s_i\ge0,\qquad
 \sum_{i<r}s_i\le\eta E .
\]
It is step-mass-saturating if $\sum_{i=1}^r\tau(m_i)\ge(1-\eta)E$.
Each $s_i$ equals the total step size of the intervening omitted states.
\end{definition}

\begin{proposition}\label{thm:variance-ladder}
Suppose $V^{-1}\sum_{j\in J}(\tau(n_j)-T)^2\le\epsilon^2T^2$ with
$0<\epsilon\le1/16$. The states with
$|\tau(n_j)-T|\le\sqrt\epsilon\,T$ form a step-mass-saturating
$3\sqrt\epsilon$-dynamic near-ladder, and $r\ge(1-\epsilon)V$.
Except for at most $\epsilon V$ adjacent pairs, their gaps equal
$\tau(m_i)=T+O(\sqrt\epsilon\,T)$.
\end{proposition}
\begin{proof}
Chebyshev's inequality leaves at most $\epsilon V$ exceptional indices.
Their total step mass is at most
\[
 T\epsilon V+\sqrt{\epsilon V}
   \left(\sum_{j\in J}(\tau(n_j)-T)^2\right)^{1/2}
 \le(\epsilon+\epsilon^{3/2})E\le2\epsilon E.
\]
The retained states therefore satisfy both mass conditions. Each gap
with positive $s_i$ contains a distinct exceptional index.
\end{proof}

By Corollary~\ref{cor:mass-away}, saturating ladders with
$\eta\to0$ can occur along a family of large crossings only when
$T=(\log N)^{\log2+o(1)}$. In particular, if $T\asymp\log N$, the
relative step variance grows without bound.

\subsection{Phase increments}\label{sec:phase}
For $q\ge2$ and $q\nmid h$ put $z_j=\exp(2\pi i h n_j/q)$ and
$u_j=\exp(-2\pi i h\tau(n_j)/q)$. Directly from the recurrence,
$z_{j+1}=z_j u_j$ and $|z_{j+1}-z_j|=|u_j-1|$.
For $q=2$, $h=1$, the latter squared quantity is exactly four times
the indicator that $n_j$ is a square.

\begin{lemma}\label{lem:phase}
For $J=\{r,\ldots,t\}$ of size $V$, if
$|V^{-1}\sum_{j\in J}z_j|\ge1-\delta$, $0\le\delta\le1$, then
$V^{-1}\sum_{j=r}^{t-1}|u_j-1|^2\le8\delta$.
With $Q=q/(q,h)$ and any residue $b\bmod Q$ one also has
\[
 \frac{\#\{j\in J:\tau(n_j)\not\equiv b\bmod Q\}}{V}
 \le\frac{Q^2}{16V}\sum_{j\in J}
 \left|u_j-\exp(-2\pi i hb/q)\right|^2.
\]
\end{lemma}
\begin{proof}
For $\bar z=V^{-1}\sum z_j$ one has
$V^{-1}\sum|z_j-\bar z|^2=1-|\bar z|^2\le2\delta$.
Bound each $|z_{j+1}-z_j|^2$ by
$2|z_{j+1}-\bar z|^2+2|z_j-\bar z|^2$ and sum.
For the second claim any two different $Q$th roots of unity are at
least $2\sin(\pi/Q)\ge4/Q$ apart.
\end{proof}

For growing $Q$, a mean-square error $o(Q^{-2})$ gives an $o(1)$
exceptional proportion. For several effective moduli $Q_i$, let $B_i$
be the indices in the same crossing $J$ where the $i$th congruence
fails. All congruences hold on $J\setminus\bigcup_i B_i$. To retain
$V-o(V)$ states one needs $|\bigcup_i B_i|=o(V)$. For a fixed number
of moduli, $|B_i|=o(V)$ for each $i$ suffices. For a growing family
one sufficient bound is $\sum_i|B_i|=o(V)$. On the common set, the
residues determine a positive step $\le M$ uniquely if
$\operatorname{lcm}_i Q_i\ge M$. Deriving such joint control from an
incidence discrepancy remains open.

\subsection{A conditional criterion for the leading constant}\label{sec:conditional}
As a complement to the unconditional results, we give two unproved
hypotheses that together imply the unweighted divisor regularity of
Proposition~\ref{prop:incidence-mean} on all scales $N=N_k\ge\sqrt x$.

\begin{hypothesis}\label{hyp:dichotomy}
There are functions $\rho(N),\eta(N)\in(0,1)$ tending to zero such
that for every sufficiently large $x$ and every such scale either
$\mathcal D_x(N)\le\rho(N)$, or the crossing contains a
step-mass-saturating $\eta(N)$-dynamic near-ladder with parameter
$T_x(N)$.
\end{hypothesis}
\begin{hypothesis}\label{hyp:antiladder}
For the same $\eta$ there is a fixed $\delta\in(0,1)$ such that every
$\eta(N)$-dynamic near-ladder with parameter $T_x(N)$ has step mass
at most $(1-\delta)E_x(N)$ for all sufficiently large $x$ and all
the indicated scales.
\end{hypothesis}
\begin{theorem}[Conditional asymptotic]\label{thm:conditional}
If Hypotheses~\ref{hyp:dichotomy} and~\ref{hyp:antiladder} hold, then
$a(x)\sim x/\log x$.
\end{theorem}
\begin{proof}
For large $x$, $\eta(N)<\delta$ on every indicated scale. The second
hypothesis excludes the ladder alternative in the first, giving
$\mathcal D_x(N)=o(1)$ uniformly. The weights $1/d$ sum to
$\frac12\log N+O(1)$ up to $\sqrt N$. Divisor pairing in
Proposition~\ref{prop:incidence-mean} gives $T\sim\log N$, so
Proposition~\ref{thm:local-global} applies with $b=1$.
\end{proof}

Giving even divisors weight $2p/d$ changes the weights in total by
\[
 |2p-1|\sum_{\substack{d\le\sqrt{2N}\\2\mid d}}\frac1d
     =|2p-1|\bigl(\tfrac14\log N+O(1)\bigr).
\]
By the triangle inequality, simultaneous regularity requires $p\to1/2$,
and then either condition implies the other. Uniform parity-dependent
regularity gives $a(x)\asymp x/\log x$, since $1/2\le1/2+p\le3/2$.
If also $p\to p_0$ uniformly on the indicated scales, the leading term is
$x/((1/2+p_0)\log x)$. For variable $p$, the crossing contributions
$E/T$ must be summed individually.

\section{Computation through one billion}\label{sec:numerics}
We evaluated \eqref{eq:hitting-time} for every integer
$1\le n\le10^9$. A segmented sieve factors the integers in each
successive block and computes $\tau(n)$ from their prime exponents.
The hitting times can then be evaluated in increasing order. Since
$\tau(n)\le2\sqrt n$, a circular buffer of length greater than
$2\sqrt{10^9}$ contains every required predecessor.
Records for $\Delta(n)$ and the integer sum $S_2(X)$ are accumulated
as the computation proceeds.

\begin{table}[htbp]
\centering
\caption{Hitting times and variation across starting values.
The mean and population standard deviation of $R(n)=a(n)\log(n)/n$
use all $X/2<n\le X$. The last column uses the full range $1\le n\le X$.}
\label{tab:intervals}
\begin{tabular}{@{}r r r r r@{}}
\toprule
$X$ & $a(X)$ & Mean of $R$ & Std.\ dev.\ of $R$ & $\sqrt{S_2(X)}/X$\\
\midrule
$10^5$ & $7\,534$ & $0.885716$ & $0.029122$ & $0.9125$\\
$10^6$ & $65\,059$ & $0.906984$ & $0.010919$ & $0.8030$\\
$10^7$ & $556\,759$ & $0.899319$ & $0.004020$ & $0.8057$\\
$10^8$ & $4\,917\,560$ & $0.905852$ & $0.001515$ & $0.8068$\\
$10^9$ & $43\,981\,442$ & $0.911331$ & $0.000611$ & $0.8695$\\
\bottomrule
\end{tabular}
\end{table}

Over $5\cdot10^8<n\le10^9$, $R(n)$ ranges from $0.910520$ at
$n=509\,633\,580$ to $0.915878$ at $n=697\,012\,939$ (rounded).
The means in Table~\ref{tab:intervals} vary between scales. The range
at the largest scale therefore gives little information about a limiting
constant. In contrast, the exact second moment at the upper endpoint is
\[
 S_2(10^9)=756\,073\,301\,209\,214\,126 .
\]

\subsection{Two adjacent records}
The largest decrease in the range $1\le n\le10^9$ is $-181\,344$,
first attained at $n=957\,098\,004$. The largest increase is
$181\,343$, first attained at the very next integer.
Figure~\ref{fig:record-window} shows the local hitting times.
These two differences also admit a check independent of the forward
recurrence computation. Count divisor pairs $(d,v/d)$ with $d\le\sqrt v$
on descending blocks and follow the three paths until they meet.
Table~\ref{tab:meeting} gives their times to a common state,
$952\,517\,718$. Beyond it all three paths have the same remaining
time, $41\,987\,420$.

\begin{table}[htbp]
\centering
\caption{Three consecutive starting values. Subtracting adjacent entries
in the last column recovers both signed records.}
\label{tab:meeting}
\begin{tabular}{@{}r r r@{}}
\toprule
Starting value $n$ & $a(n)$ & Steps to $952\,517\,718$\\
\midrule
$957\,098\,003$ & $42\,361\,769$ & $374\,349$\\
$957\,098\,004$ & $42\,180\,425$ & $193\,005$\\
$957\,098\,005$ & $42\,361\,768$ & $374\,348$\\
\bottomrule
\end{tabular}
\end{table}
\begin{figure}[htbp]
\centering
\begin{tikzpicture}
\begin{axis}[
 width=.94\linewidth,height=5.5cm,
 xmin=-132,xmax=132,xtick={-128,-64,0,64,128},
 xlabel={Starting value $n$ minus $957\,098\,004$},
 ylabel={$(a(n)-42\,180\,425)/10^3$},
 axis lines=left,axis line style={black!62},tick align=outside,
 tick label style={font=\small},label style={font=\small},
 ymajorgrids=true,grid style={black!12},scaled ticks=false,
 enlarge y limits=.13,
]
\addplot[only marks,mark=*,mark size=1.25pt,black!57]
  table[col sep=comma,x expr=\thisrow{offset}+1,y expr=\thisrow{relative_steps}/1000]{record_window.csv};
\addplot[thick,draw=black,mark=o,mark size=2.5pt,
 mark options={fill=white,solid}]
  coordinates {(-1,181.344) (0,0) (1,181.343)};
\end{axis}
\end{tikzpicture}
\caption{Hitting times at 257 consecutive starts around the record pair.
The origin is the start $n_*=957\,098\,004$ of the largest decrease.
The vertical coordinate is $(a(n)-a(n_*))/10^3$. Connected open circles
at offsets $-1,0,1$ mark the three starts in Table~\ref{tab:meeting}.}
\label{fig:record-window}
\end{figure}

\FloatBarrier

\section{Open questions}
\begin{enumerate}[leftmargin=*,label=\textup{(\roman*)}]
\item \emph{The size of the second moment.} Is
$S_2(X)=X^{2+o(1)}$? Corollary~\ref{cor:jump-variation} supplies the
lower exponent $2$. A matching upper bound would control the total
contribution of the largest differences. Another question is whether
a positive proportion of starting values exhibit jumps on a
near-square-root scale.
\item \emph{The size of the hitting time.} Close the gap in
Theorem~\ref{thm:length} and determine whether $a(x)\asymp x/\log x$.
Even if a leading term $Cx/\log x$ exists, Theorem~\ref{thm:jump-model}
limits the accuracy of this particular approximation. An average over
starting values would be another way to test for a leading constant.
\item \emph{Step mass on a crossing.} If the mean is comparable to
$\log N$, Corollary~\ref{cor:mass-away} puts almost all the displacement
on a small subset of visited states. Estimate the contribution of those
states uniformly, or quantify divisor incidence after conditioning on
the parity of the states. For the conditional reduction, the outstanding
ladder scale is $T=(\log N)^{\log 2+o(1)}$.
\end{enumerate}

\section*{Data and code availability}
The numerical supplement accompanying this version has Zenodo DOI
\href{https://doi.org/10.5281/zenodo.24418699}{\texttt{10.5281/zenodo.24418699}}
\cite{MantovanelliZenodo}. It contains the C++17 program,
machine-readable tables, and instructions for regenerating and checking
them. It includes the initial values through $10^4$, pointwise
normalizations and parity statistics through $10^7$, interval and jump
statistics through $10^9$, and the trajectories' meeting-time data.
The file used to draw Figure~\ref{fig:record-window} is also included
in the manuscript source.

\section*{AI disclosure}
The author used OpenAI’s ChatGPT for help in language editing,
LaTeX restructuring, mathematical explorations, and the preparation of verification
scripts. All mathematical statements, proofs, computations, references, and the final
presentation were independently verified by the author, who assumes full responsibility
for the content.

\bibliographystyle{plain}
\bibliography{references}
\end{document}